\documentclass[conference]{IEEEtran}
\ifCLASSINFOpdf
\else
\fi

\usepackage{amsmath}
\usepackage{amsfonts,amssymb}
\usepackage{amsthm}
\usepackage{algorithmic}
\usepackage{algorithm}

\newcommand{\N}{\mathbb{N}} 
\newcommand{\Z}{\mathbb{Z}} 
\newcommand{\Q}{\mathbb{Q}} 
\newcommand{\C}{\mathbb{C}} 
\newcommand{\R}{\mathbb{R}} 
\newcommand{\p}{\mathfrak{p}} 
\newcommand{\pg}{\mathfrak{p}} 
\newcommand{\OK}{\mathcal{O}_K}
\newcommand{\Nm}{\mathcal{N}}
\newcommand{\ag}{\mathfrak{a}}
\newcommand{\bg}{\mathfrak{b}}
\newcommand{\cg}{\mathfrak{c}}
\newcommand{\dg}{\mathfrak{d}}

\newcommand{\add}{\alpha_{d,\Delta , \gamma}}
\newcommand{\bdd}{\beta_{d,\Delta , \gamma}}

\newtheorem{theorem}{Theorem}
\newtheorem{lemma}{Lemma}
\newtheorem{proposition}{Proposition}
\newtheorem{corollary}{Corollary}
\begin{document}
%
\title{An algorithm for list decoding number field codes}

\author{\IEEEauthorblockN{Jean-Fran\c{c}ois Biasse}
\IEEEauthorblockA{Department of Computer Science\\
University of Calgary\\
2500 University Drive NW\\
Calgary, Alberta, Canada T2N 1N4\\
Email: biasse@lix.polytechnique.fr}
\and
\IEEEauthorblockN{Guillaume Quintin}
\IEEEauthorblockA{LIX\\ \'{E}cole Polytechnique\\
91128 Palaiseau, France\\
Email: quintin@lix.polytechnique.fr}}


%


\maketitle

\begin{abstract}
We present an algorithm for list decoding codewords of algebraic number field codes in polynomial time. 
This is the first explicit procedure for decoding number field codes whose construction were previously described by 
Lenstra~\cite{lenstra_nb_fld} and Guruswami~\cite{guruswami_nb_fld}. We rely on a new algorithm for computing 
the Hermite normal form of the basis of an $\OK$-module due to Biasse and Fieker~\cite{HNF_pol} where $\OK$ is the 
ring of integers of a number field $K$. 
\end{abstract}


%
\IEEEpeerreviewmaketitle

\section{Introduction}
Algorithms for list decoding Reed-Solomon codes, and their generalization the algebraic-geometric codes 
are now well understood. The codewords consist of sets of functions whose evaluation at a certain number of points 
are sent, thus allowing the receiver to retrieve them provided that the number of errors is manageable. 

The idea behind algebraic-geometric codes can be adapted to define algebraic codes whose messages are encoded 
as a list of residues redundant enough to allow errors during the transmission. The Chinese Remainder codes (CRT codes) 
have been fairly studied by the community~\cite{guruswami_sudan_soft_CRT,Mandelbaum}. The encoded messages are 
residues modulo $N:=p_1,\cdots,p_n$ of numbers $m\leq K:= p_1\cdots p_k$ where $p_1 < p_2 < \cdots < p_n$ are prime 
numbers. They are encoded by using 
\[   \left. \begin{array}{ccc}
         \Z & \longrightarrow  & \Z/p_1\times \cdots \times \Z/p_n \\
        m & \longmapsto & (m\bmod p_1,\cdots ,m \bmod p_n). \end{array} \right. \] 
Decoding algorithms for CRT codes were significantly improved to reach the same level of tolerance to errors as those 
for Reed-Solomon codes~\cite{Boneh,sudan,guruswami_sudan_soft_CRT}. As algebraic-geometric codes are a generalization of 
Reed-Solomon codes, the idea arose that we could generalize the results for CRT codes to redundant residue codes based 
on number fields. Indeed, we can easily define an analogue of the CRT codes where a number field $K$ plays the role of $\Q$ and its ring of integers $\OK$ plays the role of $\Z$. Then, for prime ideals $\p_1,\cdots,\p_n$ such that $\Nm(\p_1) < \cdots < \Nm(\p_n)$, a message $m\in\OK$ can be encoded by using
\[   \left. \begin{array}{cccc}
      &   \OK & \longrightarrow  & \OK/\p_1\times \cdots \times \OK/\p_n \\
    c:&    m & \longmapsto & (m\bmod \p_1,\cdots ,m \bmod \p_n). \end{array} \right. \] 
The construction of good codes on number fields have been independantly studied by Lenstra~\cite{lenstra_nb_fld} and 
Guruswami~\cite{guruswami_nb_fld}. They provided indications on how to chose number fields having good properties for 
the underlying codes. In particular, Guruswami~\cite{guruswami_nb_fld} showed the existance of 
asymptotically good number field codes, that is a family $\mathcal{C}_i$ of $[n_i,k_i,d_i]_q$ codes of increasing 
block length with 
$$\lim\inf \frac{k_i}{n_i} > 0 \ \text{and}\ \lim\inf\frac{d_i}{n_i} > 0.$$
Neither of them could provide a decoding algorithm. In the concluding remarks of~\cite{guruswami_nb_fld}, Guruswami 
idendifies the application of the decoding paradigm of~\cite{guruswami_phd,guruswami_sudan_reed_sol,guruswami_sudan_soft_CRT} 
to number field codes as an open problem. 

\noindent\textbf{Contribution: }The main contribution of this paper is to provide the first algorithm for decoding 
number field codes. We first show that a direct adaptation of an analogue of Coppersmith's theorem due to Cohn and 
Henninger~\cite{Coppersmith} allows to follow the approach of Boneh~\cite{Boneh} which does not allow to reach 
the Johnson bound. Then we adapt the decoding paradigm of~\cite[Chap. 7]{guruswami_phd} to number field codes, by using 
methods for manipulating modules over the ring of integers of a number field recently described in~\cite{HNF_pol} to 
achieve the Johson bound.

Throughout this paper, 
we denote by $K$ a number field of degree $d$, of discriminant $\Delta$ and of ring of integers $\OK$. The prime 
ideals $(\p_i)_{i\leq n}$ satisfy $\Nm(\p_1)<\Nm(\p_2)<\cdots<\Nm(\p_n)$, and we define $N:=\prod_{i\leq n}\Nm(\p_i)$ 
and $B := \prod_{i\leq k}\Nm(\p_i)$ for integers $k,n$ such that $0< k  < n$. Before describing our algorithm in more 
details in the following sections, let us state the main result of the paper.

\begin{theorem}
Let $\varepsilon > 0$, and a message $m\in\OK$ satisfying $\| m \| \leq B$, then there is an algorithm that returns all the messages $m'\in\OK$ such that $\| m' \|\leq B$ and that $c(m)$ and $c(m')$ have mutual agreement $t$ satisfying
$$t\geq \sqrt{k(n + \varepsilon)}.$$
This algorithm is polynomial in $d$ , $\log(N)$, $1/\varepsilon$ and $\log|\Delta|$.
\end{theorem}

\section{Generalities on number fields}

Let $K$ be a number field of degree $d$. It has $r_1\leq d$ real embeddings $(\theta_i)_{i\leq r_1}$ and $2r_2$ 
complex embeddings $(\theta_i)_{r_1 < i \leq 2r_2}$ (coming as $r_2$ pairs of conjugates). The field 
$K$ is isomorphic to $\OK\otimes\Q$ where $\OK$ denotes the ring of integers of $K$. 
We can embed $K$ in 
$$K_\R := K\otimes \R \simeq \R^{r_1}\times \C^{r_2}, $$ 
and extend the $\theta_i$'s to $K_\R$. Let $T_2$ be the Hermitian form 
on $K_\R$ defined by 
$$T_2(x,x') := \sum_i \theta_i(x)\overline{\theta_i}(x'),$$
and let $\| x\| := \sqrt{T_2(x,x)}$ be the corresponding $L_2$-norm. Let $(\alpha_i)_{i\leq d}$ such that 
$\OK = \oplus_i \Z\alpha_i$, then the discriminant of $K$ is given by $\Delta = \det^2(T_2(\alpha_i,\alpha_j))$. 
The norm of an element $x\in K$ is defined by $\Nm(x) = \prod_i|\theta_i(x)|$.

We encode our messages with prime ideals of $\OK$. However, for decoding, we need a more general notion of ideal, 
namely the fractional ideals of $\OK$. They can be defined as finitely generated $\OK$-modules of $K$. When a 
fractional ideal is contained in $\OK$, we refer to it as an integral ideal.
For every fractional ideal $I$ of $\OK$, there exists $r\in\Z$ such that $rI$ is integral. The sum and product of 
two fractional ideals of $\OK$ is given by 
\begin{align*}
IJ &= \{ i_1j_1 + \cdots + i_lj_l\mid l\in \N, i_1,\cdots i_l\in I, j_1,\cdots j_l\in J\}\\
I + J &= \{ i + j\mid i\in I , j\in J\}.
\end{align*}
The fractional ideals of $\OK$ are invertible, that is for every fractional ideal $I$, there exists 
$I^{-1}:= \{ x\in K\mid xI\subseteq \OK\}$ such that $II^{-1} = \OK$. The set of fractional ideals is 
equipped with a norm function defined by $\Nm(I) = \det(I)/\det(\OK)$. The norm of ideals is multiplicative, 
and in the case of an integral ideal, we have $\Nm(I) = |\OK / I|$. Also note that the norm of $x\in K$ is 
precisely the norm of the principal ideal $(x) = x\OK$. 

In the following, we will study finitely generated sub $\OK$-module of $\OK[y]$. Let $M\subseteq K^l$ be a 
finitely generated $\OK$-module. As in~\cite[Chap. 1]{cohen2}, we say that $[(a_i),(\mathfrak{a}_i)]_{i\leq n}$, where 
$a_i\in K$ and $\mathfrak{a}_i$ is a fractional ideal of $K$, is a pseudo-basis for $M$ if 
$M = \mathfrak{a}_1a_1\oplus \cdots \oplus \mathfrak{a}_na_n$. We also call a pseudo-matrix representing $M$ 
the matrix of the coefficients 
of the $(a_i)_{i\leq n}$ along with the ideals $\ag_i$. The algorithm~\cite[Alg.4]{HNF_pol} 
returns a pseudo-matrix representing $M$ where the matrix of the $(a_i)_{i\leq n}$ has a triangular shape in polynomial time.

\section{Decoding with Copersmith's theorem}

An analogue of Copersmith's theorem was described by Cohn and Henninger in~\cite{Coppersmith}. It was used 
to provide an elegant way of decoding Reed-Solomon codes, and the possibility to use it for breaking lattice-
based cryptosystems in $\OK$ modules was considered, although they concluded that it would not improve the 
state-of-the-art algorithms.

\begin{theorem}[Coppersmith]\label{th:coppersmith}
Let $f\in \OK[X]$ a monic polynomial of degree $l$, $0<\beta\leq 1$, $\lambda_1,\cdots , \lambda_d>0$ 
and $I\subsetneq \OK$ an ideal. 
We can find in polynomial time all the  $\omega\in\OK$ such that $|\omega|_i:= |\sigma_i(\omega)|\leq \lambda_i$ and  
$$\Nm(\operatorname{gcd}(f(\omega)\OK,I)>\Nm(I)^{\beta},$$
provided that the $\lambda_i$ satisfy 
$\prod_i \lambda_i < (2 + o(1))^{-d^2/2}\Nm(I)^{\beta^2/l}$.
\end{theorem}

Although not mentioned in~\cite{Coppersmith}, a straightforward adaptation of Theorem~\ref{th:coppersmith} 
with $\beta := \sqrt{\frac{\sum_{i\leq k}\log\Nm(\p_i)}{\sum_{i\leq n}\log\Nm(\p_i)}}$ where $0<k<n$, $I:= \prod_{i\leq n}\p_i$ and 
$\forall i, \lambda_i :=  \frac{1}{2^{n/2}}\prod_{i\leq k}\Nm(\p_i)^{1/n}$
provides a polynomial time algorithm for decoding number field codes.

\begin{theorem}
Let 
$(r_1,\cdots,r_n)\in\OK^n$ and $m\in\OK$ satisfying $\forall i,\ m = r_i\bmod\p_i$, 
then Theorem~\ref{th:coppersmith} 
applied to $f(\omega):= \omega - m$ allows to return in polynomial time a list 
of $m'\in\OK$ with $\|m'\|\leq \frac{1}{2^{n/2}}\prod_{i\leq k}\Nm(\p_i)^{1/n}$ that differ from $m$ in 
at most $e$ places where
$$e < n - \sqrt{kn\frac{\log\Nm(\pg_n)}{\log\Nm(\pg_1)}}.$$
\end{theorem}

In the rest of the paper, we present a method based on Guruswami's general framework for residue codes~\cite{guruswami_phd} 
that allows us to get rid in the dependency in $\frac{\log\Nm(\pg_n)}{\log\Nm(\pg_1)}$ in the decoding bound thus 
reaching the Johnson bound.

\section{Johnson-type bound for number fields codes}

A Johnson-type bound is a positive number $J$ depending on the distance,
the blocklength and the cardinalities of the Alphabets constituting the
code. It garanties that a ``small" number of codewords are in any sphere
of radius $J$. By ``small" number, we mean a number of codewords which is
linear in the code blocklength and the cardinality of the code.
In our case, the Johnson-type bound for number
fields codes depends only on the code blocklength and its minimal
distance, and ``small" means polynomial in
$\sum_{i = 1}^n \log \Nm(\p_i)$.

The Johnson-type bound of~\cite[Section 7.6.1]{guruswami_phd} remains valid for number field codes.
For any prime ideal $\p\subset\OK$, the quotient $\OK/\p$ is a finite
field. Thus the $i$'th symbol of a codeword comes from an alphabet of size
$\Nm(\p_i) = |\OK/\p_i|$ and~\cite[Th. 7.10]{guruswami_phd} can be
applied. Let $t$ be the least positive integer such that
$
\prod_{i=1}^t \Nm(\p_i) > \left(\frac{2B}{d}\right)^d,
$ 
where $d = [K:\Q]$ and let
$
T = \prod_{i=1}^t \Nm(\p_i).
$  
Then, by \cite[Lem. 12]{guruswami_nb_fld}, the minimal hamming distance
of the number fields code is at least $n - t + 1$. Using~\cite[Th. 7.10]{guruswami_phd}, we can show that for a given message and $\varepsilon > 0$, only a ``small" number 
of codewords satisfy 
\begin{equation}
\label{eq:johnson_hamming}
\sum_{i = 1}^n a_i > \sqrt{(t + \varepsilon)n},
\end{equation}
where $a_i = 1$ if the codeword and the message agree at the $i$-th position, $a_i = 0$ otherwise. Thus, if our list decoding algorithm returns all the codewords having at most
$n - \sqrt{(t + \varepsilon)n}$ errors
then this number is garanteed to be ``small". Therefore, the Johnson bound appears to be a good objective for our algorithm. 
Note that we would derive a different bound by using weighted distances. In particular, by using the $\log$-weighted hamming distance
i.e. $d(x,y) = \displaystyle\sum_{i:x \neq y \mod \p_i} \log \Nm(\p_i)$,
the condition would be
$
\sum_{i = 1}^n a_i \log\Nm(\p_i) >
\sqrt{(\log T + \varepsilon ) \log N}
$.

\section{General description of the algorithm}

In this section, we give a high-level description of our decoding algorithm. We follow the 
approach of the general framework described in~\cite{guruswami_phd}, making the arrangements 
required in our context. Our code is the set of $m\in\OK$ such that $\|m\|\leq B$ where 
$B = \prod_{i\leq k}\Nm(\p_i)$. We also define $N:=\prod_{i\leq n}\Nm(\p_i)$. A codeword $m$ is 
encoded via 
\[   \left. \begin{array}{ccc}
         \OK & \longrightarrow  & \OK/\p_1\times \cdots \times \OK/\p_n \\
       m & \longmapsto & (m\bmod \p_1,\cdots ,m \bmod \p_n). \end{array} \right. \] 
Let $z_1,\cdots,z_n$ be non-negative real numbers, and let $Z$ be a parameter. In this section, as 
well as in Section~\ref{sec:existence} and~\ref{sec:computation}, we assume that the $z_i$ are integers. 
We assume that we received a vector $(r_1,\cdots,r_n)\in\prod_i\OK/\p_i$. We wish to retrieve all the codewords 
$m$ such that $\sum_i a_iz_i > Z$ where $a_i=1$ if $m\bmod \p_i = r_i$ and 0 otherwise (we say that $m$ and 
$(r_i)_{i\leq n}$ have weighted agreement $Z$).

We find the codewords $m$ with desired weighted agreement by computing roots of a polynomial $c\in\OK[y]$ that satisfies 
\begin{equation}\label{eq:cond_norm}
\|m\|\leq B \Longrightarrow \|c(m)\| < F,
\end{equation} 
for an appropriate bound $F$. We choose the polynomial $c$ satisfying~\eqref{eq:cond_norm} in the ideal 
$\prod_{i\leq n}J_i^{z_i}\subseteq\OK[y]$ where
$$J_i = \{ a(y)(y-r_i) + p\cdot b(y)\mid a,b\in\OK[y]\ \text{and}\ p\in\p_i\}.$$
With such a choice of a polynomial, we necessarily have $c(m)\in\prod_i \p_i^{z_ia_i}$, where 
$a_i = 1$ if $c(m)\bmod \p_i = r_i$, $0$ otherwise. In particular, if $c(m)\neq 0$ then 
$\Nm(c(m)) \geq \prod_i\Nm(\p_i)^{z_ia_i}$. In addition, we know from the arithmetic-geometric 
inequality that $\|c(m)\| \geq \sqrt{d}\Nm(c(m))^{1/d}$. We thus know that if the weighted agreement satisfies 
\begin{equation}\label{eq:cond_weight}
\sum_{i\leq n} a_iz_i\log\Nm(\p_i) > -\frac{d}{2}\log(d) + d\log(F),
\end{equation}
which in turns implies $\sqrt{d}\left(\prod_i\Nm(\p_i)^{z_ia_i}\right)^{1/d} > F$, then $c(m)$ has 
to be zero, since otherwise it would contradict~\eqref{eq:cond_norm}.

\begin{algorithm}[H]
\caption{Decoding algorithm}
\begin{algorithmic}[1]\label{alg:decoding}
\REQUIRE $\OK$, $z_1,\cdots,z_n$, $B$, $Z$, $r_1,\cdots,r_n\in\prod_i\OK/\p_i$.
\ENSURE All $m$ such that $\sum_i a_iz_i > Z$.
\STATE Compute $l$ and $F$.
\STATE Find $c\in\prod_{i\leq n}J_i^{z_i}\subseteq\OK[y]$ of degree at most $l$ such that $\|m\|\leq B \Longrightarrow \|c(m)\| < F$.
\STATE Find all roots of $c$ and report those roots $\xi$ such that $\|\xi\|\leq B$ and $\sum_i a_iz_i > Z$.
\end{algorithmic}
\end{algorithm}

\section{Existence of the decoding polynomial}\label{sec:existence}

In this section, given weights $(z_i)_{i\leq n}$, we prove the existence of a polynomial $c\in\prod_i J_i^{z_i}$ and a constant $F>0$ such 
that for all $\|m\|\leq B$, $m\in\OK$, we have $\|c(m)\|\leq F$. This proof is not constructive. The actual 
computation of this polynomial will be described in Section~\ref{sec:computation}. We first need to estimate 
the number of elements of $\OK$ bounded by a given size. 

\begin{lemma}\label{lem:Minkowski}
Let $F'>0$ and $0<\gamma < 1$, then the number of $x\in\OK$ such that $\|x\|\leq F'$ is at least
$$\left\lfloor \frac{\pi^{d/2}F'^d}{2^{r_1+r_2-1+\gamma}\sqrt{|\Delta|}\Gamma(d/2)}\right\rfloor.$$
\end{lemma}

\begin{proof}
As in~\cite[Chap. 5]{neukirch}, we use the standard results of Minkowski theory for our purposes. 
More precisely, there is an isomorphism
$f:K_\R\longrightarrow \R^{r_1 + 2r_2}$
and a scalar product $(x,y) := \sum_{i\leq r_1}x_iy_i + \sum_{r_1<i\leq r_1+2r_2}2x_iy_i$ on $\R^{r_1+2r_2}$ 
transfering the canonical measure from $K_\R$ to $\R^{r_1+2r_2}$. Let $\lambda = f(\OK)$, $X := \{ x\in K_\R\mid \|x\|\leq F'\}$, 
and $m\in\N$. We know from Minkowski's lattice point theorem that if
$\operatorname{Vol}(X) > m2^d\det(\lambda),$
then $\#(f(x)\cap\lambda) \geq m$. As $\operatorname{Vol}(X) = 2^{r_2}\left(2\pi^{d/2}F'^d/\Gamma(d/2)\right)$ 
and $\det(\lambda) = \sqrt{|\Delta|}$, we have the desired result.
\end{proof}

Then, we must derive from Lemma~\ref{lem:Minkowski} an analogue of~\cite[Lemma 7.6]{guruswami_phd} in our 
context. This lemma allows us to estimate the number of polynomials of degree $l$ satisfying~\eqref{eq:cond_norm}. 
To simplify the expressions, we use the following notation in the rest of the paper
$$\add := \frac{\pi^{d/2}}{2^{r_1 + r2-1 + \gamma}\sqrt{|\Delta|}\Gamma(d/2)}.$$

\begin{lemma}\label{lem:nb_pol}
For positive integers $B,F'$, the number of polynomials $c\in\OK[y]$ of degree at most $l$ satisfying~\eqref{eq:cond_norm} 
is at least
$$\left( \add \left( \frac{F'}{(l+1)B^{l/2}}\right)^d\right)^{l+1}.$$ 
\end{lemma}

\begin{proof}
Let $c(y) = c_0 + c_1y + \cdots + c_ly^l$. We want the $c_i$'s to satisfy $\|c_im^i\|< F'/(l+1)$ whenever $\|m\|\leq B$. 
This is the case when $\|c_i\| < F'/(B^i(l+1))$. By Lemma~\ref{lem:Minkowski}, there are at least 
$\add\left(F'/((l+1)B^i)\right)^d$ possibilities for $c_i$. Therefore, the number of polynomials $c$ 
satisfying~\eqref{eq:cond_norm} is at least 
$$(\add)^{l+1}\left(\left(\frac{F'}{l+1}\right)^{l+1}\prod^l_{i = 0}B^{-i}\right)^d,$$
which finishes the proof. 
\end{proof}

Now that we know how to estimate the number of $c\in\OK[y]$ or degree at most $l$ satisfying~\eqref{eq:cond_norm}, 
we need to find a lower bound on $F$ to ensure that we can find such a polynomial in $\prod_iJ_i^{z_i}$. The following 
lemma is an equivalent of~\cite[Lemma 7.7]{guruswami_phd}.

\begin{lemma}
Let $l,B,F$ be positive integers, there exists $c\in\prod_i J_i^{z_i}$ satisfying~\eqref{eq:cond_norm} provided that
\begin{equation}\label{eq:cond_F}
F > 2(l+1)B^{l/2}\frac{1}{(\add)^{1/d}}\left(\prod_i\Nm(\p_i)^{\binom{z_i+1}{2}}\right)^{\frac{1}{d(l+1)}}.
\end{equation}
\end{lemma} 

\begin{proof}
Let us apply Lemma~\ref{lem:nb_pol} to $F' = F/2$. There are at least 
$$\left( \add \left( \frac{F/2}{(l+1)B^{l/2}}\right)^d\right)^{l+1}$$
polynomial $c\in\OK[y]$ satisfying $\|m\|\leq B\Rightarrow \|c(m)\| < F/2$. In addition, we know 
from~\cite[Corollary 7.5]{guruswami_phd} that $\prod_i\left|\Nm(\p_i)\right|^{\binom{z_i+1}{2}}\geq \left|\OK[y]/\prod_i J_i^{z_i}\right|$, which implies that if~\eqref{eq:cond_F} is satisfied, then necessarily
$$\left( \add \left( \frac{F/2}{(l+1)B^{l/2}}\right)^d\right)^{l+1} > \left|\OK[y]/\prod_i J_i^{z_i}\right|.$$
This means that there are at least two distinct polynomials $c_1,c_2\in\OK[y]$ of degree at most $l$ such 
that $(c_1-c_2)\in\prod_i J_i^{z_i}$ and $\|c_1(m)\|,\|c_2(m)\| < F/2$ whenever $\|m\|\leq B$. The choice 
of $c:=c_1-c_2$ finishes the proof.
\end{proof}

\section{Computation of the decoding polynomial}\label{sec:computation}

Let $l>0$ be an integer to be determined later. To compute $c\in\prod_i J_i^{z_i}$ of degree 
at most $l$ satisfying~\eqref{eq:cond_norm}, we need to find a short pseudo-basis of the sub $\OK$-module 
$M\cap\prod_i J_i^{z_i}$ 
of $K^{l+1}$ where $M$ is the $\OK$-module of the elements of $\OK[y]$ of degree at most $l$ embedded in 
$K^{l+1}$ via $\sum_i c_iy^i\rightarrow (c_i)$.
We first compute a peudo-generating set for each $M\cap J_i^{z_i}$, then we compute a pseudo-basis for their 
intersection, and we finally call the algorithm of~\cite{stehle_fieker_LLL} to produce a short peudo-basis 
of $M\cap\prod_i J_i^{z_i}$ from which we derive $c$. 

An algorithm for computing a pseudo-basis of the intersection of two modules given by their pseudo basis is 
described by Cohen in~\cite[1.5.2]{cohen2}. It relies on the HNF algorithm for $\OK$-modules. The HNF algorithm 
presented in~\cite[1.4]{cohen2} is not polynomial, but a variant recently presented in~\cite{HNF_pol} enjoys 
this property. We can therefore apply~\cite[1.5.2]{cohen2} with the HNF of~\cite{HNF_pol} successively for each 
pseudo-basis of $M\cap J_i^{z_i}$ to produce a pseudo-basis of $M\cap \prod_i J_i^{z_i}$.

\begin{algorithm}[ht]
\caption{Computation of the decoding polynomial}
\begin{algorithmic}[1]\label{alg:comp_c}
\REQUIRE $(\p_i,z_i)_{i\leq n}$, $l$, $N$, $B$, $F$ such that $\exists c\in \prod_iJ_i^{z_i}$ of degree at most $l$ satisfying~\eqref{eq:cond_norm} for $F$, and the encoded message $(r_1,\cdots,r_n)\in\prod_i\OK/\p_i$. 
\ENSURE $c\in \prod_iJ_i^{z_i}$ satisfying~\eqref{eq:cond_norm} for $F' = 2^{\frac{dl}{2}}\sqrt{l+1}\left(2^{2 + d\left(6 + 3d\right)} d^3|\Delta|^{2 + \frac{11}{2d}}\right) F$ of degree at most $l$.
\FOR {$i\leq n$}
\STATE $\tilde{z_i}\leftarrow \min(z_i,l)$.
\STATE For $0\leq j\leq \tilde{z_i}$: $\ag^i_j\leftarrow \p_i^{z_i-j}$, $a^i_j\leftarrow (y-r_i)^j$.
\STATE For $1\leq j\leq l-z_i$: $\ag^i_j\leftarrow \OK$,  $a^i_j\leftarrow y^j(y-r_i)^{z_i}$.
\STATE Let $\left((a^i_j),(\ag^i_j)_{j\leq l+1}\right)$ be a pseudo matrix for $M\cap J_i^{z_i}$.
\ENDFOR
\STATE Compute a pseudo-basis $[(c_i),(\cg_i)]_{i\leq l+1}$ of $M_1=M\cap \prod_i J_i^{z_i}$.
\STATE Deduce a pseudo basis $[(d_i),(\dg_i)]_{i\leq l+1}$ of the module $M_2$ given by
$$(v_0, v_1,\cdots,v_l)\in M_1\Longleftrightarrow (v_0, v_1\cdot B,\cdots,v_l\cdot(B)^l)\in M_2.$$
\STATE Let $[(b_i),(\bg_i)]_{i\leq l+1}$ be a short peudo-basis of $M_2$ obtained with the reduction algorithm of~\cite{stehle_fieker_LLL}.
\STATE Let $x_1,x_2$ be a short basis of $\bg_1$ obtained with~\cite[Th. 3]{stehle_fieker_LLL}.
\RETURN $c\in M_1$ corresponding to $x_1b_1\in M_2$.
\end{algorithmic}
\end{algorithm}

\section{Good weight settings}\label{sec:weight}

To derive our main result, we need to consider weights $z_i > 0$ in $\R$ rather than $\Z$. Let 
$$\bdd := \frac{d^{3-\frac{d}{2}}2^{3\left(1 + d(2 + d)\right)}|\Delta|^{2 + \frac{11}{2d}}}{\left.\add\right.^{\frac{1}{d}}},$$
then by combining~\eqref{eq:cond_weight}, \eqref{eq:cond_F} and Algorithm~\ref{alg:comp_c}, we 
know that given $(r_1,\cdots,r_n)\in\prod_{i\leq n}\OK/\p_i$, $l>0$, $B = \prod_{i\leq k}\Nm(\p_i)$ and 
integer weights $z_i > 0$, Algorithm~\ref{alg:comp_c} returns a polynomial $c$ of degree at most $l$ such 
that all $m\in\OK$ satisfying $\|m\|\leq B$ and
\begin{align}
\sum_{i\leq n} a_iz_i&\log\Nm(\p_i) \geq \frac{l}{2}\log(2^{d^2}B^d) + \frac{3d}{2}\log(l+1) \nonumber \\
& + \frac{1}{l+1}\sum_{i\leq n}\binom{z_i+1}{2}\log\Nm(\p_i) + \log\bdd, \label{eq:dec_cond_fin}
\end{align}
(where $a_i = 1$ if $m\bmod \p_i = r_i$, 0 otherwise) are roots of $c$. In the following, we no longer assume 
the $z_i$ to be integers. However, we will use our previous results with the integer weights $z^*_i := \lceil Az_i\rceil$ for 
a sufficently large integer $A$ to be determined.

\begin{proposition}\label{prop:dec_final}
Let $\varepsilon > 0$, non-negative reals $z_i$, $B = \prod_{i\leq k}\Nm(\p_i)$, and an encoded message 
$(r_1,\cdots,r_n)\in\prod_i\OK/\p_i$, then our algorithm finds all the $m\in\OK$ such that $\|m\|\leq B$ and 
\begin{small}
$$\sum_{i\leq n}a_iz_i\log\Nm(\p_i)\geq \sqrt{\log(2^{d^2}B^d)\left(\sum_{i\leq n}z^2_i\log\Nm(\p_i) + \varepsilon z^2_{max}\right)},$$
\end{small}
where $a_i = 1$ if $m\bmod \p_i = r_i$, 0 otherwise.
\end{proposition}

\begin{proof}
Note that we can assume without loss of generality that $z_{max} = 1$. Let $z^*_i = \lceil Az_i\rceil$ for a 
sufficently large integer $A$, which thus satisfies $Az_i\leq z^*_i < Az_i + 1$. The decoding condition~\eqref{eq:dec_cond_fin} 
is met whenever 
\begin{align}
\sum_{i\leq n} a_iz_i&\log\Nm(\p_i) \geq \frac{l}{2A}\log(2^{d^2}B^d) + \frac{3d}{2A}\log(l+1) \nonumber \\
& + \frac{A}{2(l+1)}\sum_{i\leq n}\left(z^2_i + \frac{3}{A}z_i + \frac{2}{A^2}\right)\log\Nm(\p_i)\nonumber \\
&+ \frac{1}{A}\log\bdd. \label{eq:dec_cond_A}
\end{align}

Let $Z_i := z^2_i + \frac{3}{A}z_i + \frac{2}{A^2}$ for $i\leq n$ and
$$l := \left\lceil A\sqrt{\frac{\sum_{i\leq n}Z_i \log\Nm(\p_i)}{\log(2^{d^2}B^d)}}\right\rceil - 1.$$
We assume that $A\geq \log(2^{d^2}B^d)$, which ensures that $l > 0$. For this choice of $l$, condition~\eqref{eq:dec_cond_A} 
is satisfied whenever 
\begin{align}
\sum_{i\leq n} a_iz_i\log\Nm(\p_i) \geq & \frac{3d}{2A}\log\left(A\sqrt{\frac{\sum_{i\leq n}Z_i \log\Nm(\p_i)}{\log(2^{d^2}B^d)}}+1\right) \nonumber \\
& + \sqrt{\log(2^{d^2}B^d)\left(\sum_{i\leq n}Z_i\log\Nm(\p_i)\right)} \nonumber \\
&+ \frac{1}{A}\log\bdd. \label{eq:dec_cond_l}
\end{align}
Assume that $A \geq \frac{10\log N}{\varepsilon}$ and $A\geq \frac{\log\bdd}{\log N}$, then for $N$ large enough, 
the right side of~\eqref{eq:dec_cond_l} is at most
\begin{align*}
&O\left(\frac{\log\log N}{\log N}\right) +\sqrt{\log(2^{d^2}B^d) \left(\sum_{i\leq n}z^2_i\log\Nm(\p_i) + \frac{\varepsilon}{2}\right)}\\
&\leq \sqrt{\log(2^{d^2}B^d) \left(\sum_{i\leq n}z^2_i\log\Nm(\p_i) + \varepsilon\right)}
\end{align*}
The degree $l$ of our decoding polynomial $c$ is therefore polynomial in $\log N$, $\frac{1}{\varepsilon}$, $d$ 
and $\log|\Delta|$. By~\cite[2.3]{ayad}, we know that the complexity to find the roots of $c$ is polynomial in $d$, $l$ and 
in the logarithm of the height of $c$, which we already proved to be polynomial in the desired values.

\end{proof}

\begin{corollary}
Let $\varepsilon > 0$, $k < n$ and prime ideals $\p_1,\cdots\p_n$ satisfying $\Nm(\p_i) < \Nm(\p_{i+1})$ and $\log\Nm(\p_{k+1}) \geq \max(2dk\log\Nm(\p_k) , 2d^2)$, then with the previous notations, our algorithm finds a list of all codewords which agree with a received word in $t$ places provided $t\geq\sqrt{k(n+\varepsilon)}$.
\end{corollary}

\begin{proof}
The proof is similar to the one of~\cite[Th. 7.14]{guruswami_phd}. The main difference is that we define 
$\delta := k - \frac{\log(2^{d^2}B^d)}{\log\Nm(\p_{k+1})}$ which satisfies $\delta \geq 0$ since by assumption 
$\log\Nm(\p_{k+1}) \geq \max(2dk\log\Nm(\p_k) , 2d^2)$. We apply Proposition~\ref{prop:dec_final} with 
$z_i = 1/\log \Nm(\p_i)$ for $i\geq k+1$, $z_i = 1/\log \Nm(\p_{k+1})$ for $i\leq k$, and 
$\varepsilon' = \varepsilon /\log\Nm(\p_{k+1})$. It allows us to retrieve the codewords whose number of 
agreements $t$ is at least 
\small
\begin{align*}
&\sqrt{\frac{\log(2^{d^2}B^d)}{\log\Nm(\p_{k+1})}\left( \frac{\log(B)}{\log\Nm(\p_{k+1})} + \sum^n_{i =k+1}\frac{\Nm(\p_{k+1})}{\log\Nm(\p_i)} + \varepsilon'\right)}\\
&\leq \delta + \sqrt{\frac{\log(2^{d^2}B^d)}{\log\Nm(\p_{k+1})}\left( \frac{\log(2^{d^2}B^d)}{\log\Nm(\p_{k+1})} + \sum^n_{i =k+1}\frac{\Nm(\p_{k+1})}{\log\Nm(\p_i)} + \varepsilon \right)}.
\end{align*}
\normalsize
This condition is met whenever $t\geq \delta + \sqrt{(k-\delta)(n-\delta+\varepsilon)}$. From the Cauchy-Schwartz 
inequality, we notice that 
$$\sqrt{k(n+\varepsilon)} \geq \sqrt{(k-\delta)(n-\delta+\varepsilon)},$$
which proves that our decoding algorithm works when $t\geq\sqrt{k(n+\varepsilon)}$.
\end{proof}

\section{Conclusion}

We presented the first method for list decoding number field codes. A straightforward 
application of Theorem~\ref{th:coppersmith} allows to derive a decoding algorithm in polynomial 
time. However, we cannot achieve the Johnson bound with this method. To solve this problem, we described 
an analogue of the CRT list decoding algorithm for codes based on number fields. This is the first algorithm 
allowing list decoding of number field codes up to the Johnson bound. We followed the approach 
of~\cite[Ch. 7]{guruswami_phd} that provides a general 
frameworks for list decoding of algebraic codes, along with its application to CRT codes. The modifications 
to make this strategy efficient in the context of number fields are substantial. We needed to refer to the 
theory of modules over a Dedekind domain, and carefully analyse the process of intersecting them, as well as 
finding short elements. We proved that our algorithm is polynomial in the size of the input, that is in 
$d$, $\log(N)$, $\log|\Delta|$ and $\frac{1}{\varepsilon}$.

\section*{Acknowledgment}

The first author would like to thank Guillaume Hanrot for his helpful comments 
on the approach based on Coppersmith's theorem.



%
\bibliographystyle{IEEEtran}
\bibliography{paper}



\end{document}